\documentclass{article}
\usepackage{authblk}
\usepackage[utf8]{inputenc}
\usepackage{amsmath}
  \usepackage{paralist}
  \usepackage{graphics} 
\usepackage{color,enumitem,graphicx}
\usepackage[colorlinks=true,urlcolor=blue,citecolor=red,linkcolor=blue,linktocpage,pdfpagelabels,bookmarksnumbered,bookmarksopen]{hyperref} 
  \usepackage{latexsym}
\usepackage{amssymb}
\usepackage{amsthm}
\usepackage{epsfig}
\everymath=\expandafter{\the\everymath\displaystyle}

\usepackage{graphicx}
\usepackage{amssymb}
 \usepackage{graphics} 
  \usepackage{epsfig} 
 \usepackage[colorlinks=true]{hyperref}
\hypersetup{urlcolor=blue, citecolor=red}

\usepackage[numbers,sort&compress]{natbib}

\usepackage{natbib}
\usepackage{graphicx}
\newtheorem{theorem}{Theorem}[section]

\newtheorem{lemma}[theorem]{Lemma}
\newtheorem{proposition}[theorem]{Proposition}

\newtheorem{remark}{Remark}

\usepackage{lineno}

\title{Fixed Point Theorem: Variants, Affine Context and Some Consequences}

\author[1]{Anderson L. A. de Araujo}
\author[1]{Edir J. F. Leite}

\affil[1]{Universidade Federal de Vi\c{c}osa, Departamento de  Matem\'atica\\ 

Av. Peter Henry Rolfs, s/n, Vi\c{c}osa, MG, Brazil, CEP 36570-900\\

E-mail: {\tt anderson.araujo@ufv.br, edirjrleite@ufv.br}}

\date{}                     
\begin{document}
\setpagewiselinenumbers

\maketitle
\begin{abstract}

\end{abstract}

In this work, we will present variants Fixed Point Theorem for the affine and classical contexts, as a consequence of general Brouwer’s Fixed Point Theorem. For instance, the affine results will allow working on affine balls, which are defined through the affine $L^{p}$ functional $\mathcal{E}_{p,\Omega}^p$  introduced by Lutwak, Yang and Zhang in the work \textit{Sharp affine $L_p$ Sobolev inequalities}, J. Differential Geom. 62 (2002), 17-38 for $p > 1$ that is non convex and does not represent a norm in $\mathbb{R}^m$. Moreover, we address results for discontinuous functional at a point. As an application, we study critical points of the sequence of affine functionals $\Phi_m$ on a subspace $W_m$ of dimension $m$ given by	
\[
		\Phi_m(u)=\frac{1}{p}\mathcal{E}_{p, \Omega}^{p}(u) - \frac{1}{\alpha}\|u\|^{\alpha}_{L^\alpha(\Omega)}- \int_{\Omega}f(x)u dx,
	\]
	 where $1<\alpha<p$, $[W_m]_{m \in \mathbb{N}}$ is dense in $W^{1,p}_0(\Omega)$ and $f\in L^{p'}(\Omega)$, with $\frac{1}{p}+\frac{1}{p'}=1$.  \\


\section{Introduction}\label{S1}   

Let $\Omega\subset\mathbb{R}^n$, $n\geq 2$ be a bounded open subset. The affine $L^{p}$ functional (or energy) for functions $u \in W^{1,p}_0(\Omega)$ is given by
$$
\mathcal{E}_{p,\Omega}(u):= \gamma_{n,p} \left( \int_{\mathbb{S}^{n-1}} \Vert\nabla_\xi u\Vert_{L^p(\Omega)}^{-n}\ d\sigma(\xi)\right)^{-\frac{1}{n}},
$$
where $\gamma_{n,p} = \left( 2 \omega_{n+p-2} \right)^{-1} \left(n \omega_n \omega_{p-1}\right) \left(n \omega_n\right)^{p/n}$. In this paper, $\nabla_\xi u(x)$ represents the directional derivative $\nabla u(x) \cdot \xi$ with respect to the direction $\xi \in \mathbb{S}^{n-1}$ and $\omega_\kappa$ is the volume of the unit Euclidean ball in $\mathbb{R}^\kappa$. 

Below are some important properties related to this functional on $W^{1,p}_0(\Omega)$, with $p>1$. Namely:

\begin{itemize}
\item[(I)] ${\cal E}^p_{p,\Omega}$ does not define a norm in $W^{1,p}_0(\Omega)$ (see Proposition 4.1 of \cite{LM});
\item[(II)] The affine $L^p$ energy ${\cal E}^p_{p,\Omega}$ is non convex on $W^{1,p}_0(\Omega)$ (see Proposition 4.1 of \cite{LM});
\item[(III)] The affine ball $\overline{B}_\varrho(0):=\{u \in W^{1,p}_0(\Omega); {\cal E}_{p,\Omega}(u)\leq \varrho\}$ is compactly immersed into $L^s(\Omega)$ for any $1 \leq s < p^*$, where $p^*:=\frac{np}{n-p}$ in case $1 < p < n$, and for any $s \geq 1$, in case $p \geq n$, (see Theorem 6.5.3 of
\cite{T1}) and is unbounded in $W^{1,p}_0(\Omega)$ (see Theorem 2 of \cite{HJM5});
\item[(IV)] The affine $L^p$ functional $\mathcal{E}_{p,\Omega}^p$ is strongly continuous and Fréchet differentiable in $W^{1,p}_0(\Omega)$ (see Theorem 1.1 of \cite{LM2});
\item[(V)] The affine $L^p$ energy $\mathcal{E}_{p,\Omega}^p$ is $C^1$ differentiable in $W^{1,p}_0(\Omega)\backslash\{0\}$ (see Theorem 1.1 of \cite{LM2}).
\end{itemize}

In 2002, the affine $L^p$ energy $\mathcal{E}_{p,\Omega}$ was introduced in \cite{LYZ2} for $p > 1$. In the sequence several related results were developed (see for example in \cite{NHJM, HS, HJM1, HJM3, HJM4, HJM5, HJS, LM, LM2, LXZ, LYZ2, LYZ3, T1}). For $p=1$, we refer to \cite{LM1} and references therein.

Let $\mathcal{B}=\{w_1,w_2,\dots\}$ be a Schauder basis of $W^{1,p}_0(\Omega)$  (see \cite{FJN,Brezis}). For each $m\geq 1$, let 
\[
W_m:=[w_1,w_2,\dots,w_m]
\]  be the $m$-dimensional subspace of $W^{1,p}_0(\Omega)$ generated by $\{w_1,w_2,\dots,w_m\}$  with norm  induced from  $W^{1,p}_0(\Omega)$ and consider $\zeta=(\zeta_1,\zeta_2,\dots,\zeta_m) \in \mathbb{R}^m$. We defined the functions 
\[
\lfloor\zeta\rfloor_m:=\mathcal{E}_{p,\Omega}\left(\sum_{j=1}^{m}\zeta_jw_j\right)
\]
and
\[
\|\zeta\|_{1,p,m}:=\Vert\sum_{j=1}^{m}\zeta_jw_j\Vert_{W^{1,p}_0(\Omega)}.
\]

It is clear that $\|\cdot\|_{1,p,m}$ define a norm in $\mathbb{R}^m$ (see \cite{AL} for the details).

By using the above notation, we can identify the spaces $(W_m, \|\cdot \|_{W^{1,p}_0(\Omega)})$ and $(\mathbb{R}^m,\|\cdot\|_{1,p,m})$ by the isometric linear
	transformation
	\begin{equation}\label{transf}
	u=\sum_{j=1}^{m}\zeta_{j}w_{j}\in W_m\mapsto
	\zeta=(\zeta_{1},\ldots,\zeta_{m})\in\mathbb{R}^{m}.
	\end{equation}

Below are some fundamental inequalities for the development of the paper: We start with the affine $L^q$ Poincaré-Sobolev inequality on $W^{1,p}_0(\Omega)$, $q\in [1,p^*]$ (see inequality (4) in \cite{LM}): There is an optimal constant $\mu_{p,q}^{\cal A} = \mu_{p,q}^{\cal A}(\Omega)>0$ such that

\begin{equation} \label{APSob}
\mu_{p,q}^{\cal A} \Vert u \Vert_{L^q(\Omega)} \leq {\cal E}_{p,\Omega}(u).
\end{equation}

Now, by Theorem 9 of \cite{HJM5}, we have
\[
C_{n,p}(\Omega)\Vert u\Vert_{L^p(\Omega)}^{\frac{n-1}{n}}\Vert\nabla u\Vert_{L^p(\Omega)}^{\frac{1}{n}}\leq\mathcal{E}_{p,\Omega}(u)\leq\Vert\nabla u\Vert_{L^p(\Omega)}
\]	
for all $u\in W^{1,p}_0(\Omega)$, where $C_{n,p}(\Omega)$ is a positive constant. For a proof of second inequality see for example to page 33 in \cite{LYZ2}. Thus, there exists a constant $C=C(n,p,m,\Omega)>0$ such that
\begin{equation}\label{ine1}
C\Vert\nabla u\Vert_{L^p(\Omega)}\leq\mathcal{E}_{p,\Omega}(u)\leq\Vert\nabla u\Vert_{L^p(\Omega)}
\end{equation}
for all $u\in W_m$. It is clear that the norms $\Vert\nabla \cdot\Vert_{L^p(\Omega)}$, $\Vert \cdot\Vert_{L^p(\Omega)}$ and $\Vert\nabla_\xi \cdot\Vert_{L^p(\Omega)}$ are equivalents in $W_m$ for every $\xi\in \mathbb{S}^{n-1}$. Moreover, by Lemma 1 of \cite{HJM5}, there is a positive constant $D_1>0$ such that
\[
D_1\Vert u\Vert_{L^p(\Omega)}\leq \Vert\nabla_\xi u\Vert_{L^p(\Omega)}\leq\Vert\nabla u\Vert_{L^p(\Omega)}
\]
for every $\xi\in \mathbb{S}^{n-1}$ and for all $u\in W^{1,p}_0(\Omega)$. Thus, there exists $D_2>0$ independent of $\xi$, such that

\begin{equation}\label{ine2}
D_2\Vert \nabla u\Vert_{L^p(\Omega)}\leq \Vert\nabla_\xi u\Vert_{L^p(\Omega)}\leq\Vert\nabla u\Vert_{L^p(\Omega)}
\end{equation}
for every $\xi\in \mathbb{S}^{n-1}$ and for all $u\in W_m$. From inequalities (\ref{ine1}) and (\ref{ine2}), we get

\begin{equation}\label{ine3}
C\Vert\nabla_\xi u\Vert_{L^p(\Omega)}\leq\mathcal{E}_{p,\Omega}(u)\leq D_3\Vert\nabla_\xi u\Vert_{L^p(\Omega)}
\end{equation}
for every $\xi\in \mathbb{S}^{n-1}$ and for all $u\in W_m$, where $C,D_3>0$ are independents of $\xi$.

To prove our main result, we will need to understand some results of the affine theory in finite dimension. In this sense, we concentrate attention on the following central ingredients. Namely:

\begin{itemize}
\item[(i)] $\lfloor\cdot\rfloor_m$ does not define a norm in $\mathbb{R}^m$;
\item[(ii)] The affine ball $\overline{B}^m_\varrho(0):=\{z \in \mathbb{R}^m; \lfloor z\rfloor_{m}\leq \varrho\}$ is a non convex subset of $\mathbb{R}^m$;
\item[(iii)] The affine ball $\overline{B}^m_\varrho(\zeta_0):=\{z \in \mathbb{R}^m; \lfloor z-\zeta_0\rfloor_{m}\leq \varrho\}$ is a closed and bounded subset of $\mathbb{R}^m$;
\item[(iv)] $\overline{B}^m_\varrho(\zeta_0)$ and $\overline{B}^{p,m}_1(0):=\{z \in \mathbb{R}^m; \|z\|_{1,p,m}\leq 1\}$ are homeomorphic in $(\mathbb{R}^m, \|\cdot\|_{1,p,m})$ (see Lemma \ref{lemhomeo} in Section \ref{s1}).
\end{itemize}

Note that (i) and (ii) follow of isometric linear transformation (\ref{transf}) and Proposition 4.1 of \cite{LM}. Note also that $\lfloor\cdot\rfloor_m$ does not satisfy the triangular inequality in $\mathbb{R}^m$. For (iii) applying the weak lower semi-continuity of the functional $u\in W^{1,p}_0(\Omega)\rightarrow {\cal E}_{p,\Omega}^p (u)$ (see Theorem 2.1 of \cite{LM}), we obtain $\overline{B}^m_\varrho(\zeta_0)$ is a closed subset of $\mathbb{R}^m$. Now, the boundedness follows from the inequality (\ref{ine1}). 

Our main objective in this work is to obtain a topological tool to be able to attack problems in the affine and classical contexts. Our first result is an affine Fixed Point Theorem. Namely:
	
	\begin{theorem}\label{prop1}
		Let $F: (\mathbb{R}^m, \|\cdot\|_{1,p,m}) \rightarrow (\mathbb{R}^m, \|\cdot\|_{1,p,m})$ be a continuous function such that $\left\langle F(\zeta),\zeta\right\rangle\geq 0$ for every $\zeta \in \mathbb{R}^m$ with $\lfloor\zeta\rfloor_{m}=\varrho$ for some $\varrho>0$,  and $\langle \cdot, \cdot \rangle ^{1/2}=|\cdot|_2$, where $|x|_2$ denote the usual euclidean norm in $\mathbb{R}^m$. Then, there exists $z_0$ in the closed affine ball $\overline{B}^m_\varrho(0)=\{z \in \mathbb{R}^m; \lfloor z\rfloor_{m}\leq \varrho\}$ such that $F(z_0)=0$.
	\end{theorem}
	
	This result is well known when the function involved is a norm and the ball is a convex subset of $\mathbb{R}^m$ (see \cite{Araujo-F2} and \cite{k}). As a counterpart of the classical context (see for instance in \cite{af1, Araujo-F2, AL, s}  and references therein), Theorem \ref{prop1} associated with the Galerkin method will be very useful for solving affine problems. We refer to \cite{HJM5, LM, LM2} for papers focused on affine problems via variational methods.
	
	Our second result is a affine Fixed Point Theorem for discontinuous functional at a point, that is, a variant of the Theorem \ref{prop1}. Precisely:

	\begin{theorem}\label{prop2}
		Let $F: (\mathbb{R}^m\setminus\{y_0\}, \|\cdot\|_{1,p,m}) \rightarrow (\mathbb{R}^m, \|\cdot\|_{1,p,m})$ be a continuous function such that $\left\langle F(\zeta),\zeta-\zeta_0\right\rangle\geq 0$ for every $\zeta \in \mathbb{R}^m\setminus\{y_0\}$ with $\lfloor\zeta-\zeta_0\rfloor_{m}=\varrho$ for some $0<\varrho<\lfloor \zeta_0-y_0\rfloor_{m}$  and $\langle \cdot, \cdot \rangle ^{1/2}=|\cdot|_2$. Then, there exists $z_0$ in the closed affine ball $\overline{B}^m_\varrho(\zeta_0)=\{z \in \mathbb{R}^m; \lfloor z-\zeta_0\rfloor_{m}\leq \varrho\}$ such that $F(z_0)=0$.
	\end{theorem}
	
	Now, our third result is a new variant of classical Fixed Point Theorems, that is, when the function involved represents a norm. Namely:
	
	\begin{theorem}\label{prop3}
		Let $\Vert\cdot\Vert_m$ be a general norm in $\mathbb{R}^m$ and $F: (\mathbb{R}^m\setminus\{y_0\}, \|\cdot\|_{1,p,m}) \rightarrow (\mathbb{R}^m, \|\cdot\|_{1,p,m})$ be a continuous function such that $\left\langle F(\zeta),\zeta-\zeta_0\right\rangle\geq 0$ for every $\zeta \in \mathbb{R}^m\setminus\{y_0\}$ with $\Vert\zeta-\zeta_0\Vert_{m}=\varrho$ for some $0<\varrho<\Vert \zeta_0-y_0\Vert_{m}$  and $\langle \cdot, \cdot \rangle ^{1/2}=|\cdot|_2$. Then, there exists $z_0$ in the closed ball $\overline{B}_\varrho(\zeta_0):=\{z \in \mathbb{R}^m; \Vert z-\zeta_0\Vert_{m}\leq \varrho\}$ such that $F(z_0)=0$.
	\end{theorem}
	
	Note that, if $F: (\mathbb{R}^m, \|\cdot\|_{1,p,m}) \rightarrow (\mathbb{R}^m, \|\cdot\|_{1,p,m})$ is a continuous function, then Theorems \ref{prop2} and \ref{prop3} hold for all $\varrho>0$, that is, we can remove the upper bound of $\varrho$ required on the hypotheses of these results.
	
	For each positive integer $m$, we define the affine functional $\Phi_m:W_m\rightarrow\mathbb{R}$ given by
	\begin{equation}\label{funct.}
		\Phi_m(u)=\frac{1}{p}\mathcal{E}_{p, \Omega}^{p}(u) - \frac{1}{\alpha}\|u\|^{\alpha}_{L^\alpha(\Omega)}- \int_{\Omega}f(x)u dx,
	\end{equation}
	whose Gâteaux derivative $\Phi_m':W_m\rightarrow W_m'$ is such that (see Theorem 10 of \cite{HJM5})
	\begin{eqnarray*}
\langle \Phi_m'(u),\varphi\rangle&=&	\displaystyle \int_{\Omega} H_{u}^{p-1}\left(\nabla u\right) \nabla H_{u}\left(\nabla u\right) \cdot \nabla \varphi d x 	- \int_{\Omega}|u|^{\alpha - 2}u\varphi dx - \int_{\Omega}f(x)\varphi dx
\end{eqnarray*}
for each $\varphi \in W_m$, where $1<\alpha<p$, 
\[
H_{u}^p(\varsigma) = \gamma_{n,p}^{-\frac np}\; {\cal E}_{p, \Omega}^{n + p}(u) \int_{\mathbb{S}^{n-1}} \left( \int_{\Omega} | \nabla_\xi u(x) |^p\, dx \right)^{-\frac{n+p}{p}}  |\langle \xi, \varsigma \rangle|^p\, d\sigma(\xi)\ \ {\rm for}\ \varsigma \in \mathbb{R}^n\, 
\]
and $f\in L^{p'}(\Omega)$, with $f\neq 0$ in $\Omega$ and $\frac{1}{p}+\frac{1}{p'}=1$. Note that

$\int_{\Omega}H_{u}^{p-1}\left(\nabla u\right) \nabla H_{u}\left(\nabla u\right) \cdot \nabla \varphi d x =$
\[
\gamma_{n,p}^{-n}\; \mathcal{E}_{p, \Omega}^{n + p}(u) \int_{\mathbb{S}^{n-1}} \left( \int_{\Omega} | \nabla_\xi u(x) |^p\, dx \right)^{-\frac{n+p}{p}}\int_{\mathbb{R}^n} \{\nabla_\xi u\}^{p-1}\langle \nabla \varphi,\xi\rangle dxd\xi,
\]
where $\{x\}^p:= \vert x\vert^p sg(x)$ and so 
\[
\displaystyle \int_{\Omega} H_{u}^{p-1}\left(\nabla u\right) \nabla H_{u}\left(\nabla u\right) \cdot \nabla u d x=\mathcal{E}_{p, \Omega}^{p}(u)
\]
for all  $u \in W_0^{1,p}(\Omega)$ (see proof of Theorem 10 of \cite{HJM5}). Notice also that $[W_m]_{m \in \mathbb{N}}$ is dense in $W^{1,p}_0(\Omega)$.
		
As a application of the Theorem \ref{prop1}, we obtain a nontrivial critical point $u_m$ of $\Phi_m$ for each $m\geq 1$. Moreover, the sequence $\mathcal{E}_{p, \Omega}^{p}(u_m)$ is bounded and so $(u_m)$ admits a strongly convergence subsequence in $L^s(\Omega)$ once $s < p^*$.

	\section{Proof of Theorems \ref{prop1}, \ref{prop2} and \ref{prop3}}\label{s1}
	
	Recently, Leite and Montenegro \cite{LM2} showed that $\mathcal{E}_{p,\Omega}^p$ is strongly continuous in $W^{1,p}_0(\Omega)$. Evidently, the affine $L^{p}$ functional $\mathcal{E}_{p,\Omega}^p$ is continuous in $W_m$, for each $m\geq 1$. To make this work more complete, we will make a proof below following the same ideas, but a little more direct, since we are in a finite dimensional space.
	
	\begin{lemma}\label{lema1} Let $m\in\mathbb{N}$. The affine $L^{p}$ functional $\mathcal{E}_{p,\Omega}$ is continuous in $W_m$.	
	\end{lemma}
	
	\begin{proof}
	Let $u_k \rightarrow u$ in $W_m$. For $u = 0$, the inequality \eqref{ine1} implies that $\mathcal{E}_{p,\Omega}(u_k)\rightarrow 0=\mathcal{E}_{p,\Omega}(0)$.

Suppose now that $u \neq 0$. Note that, by inequality \eqref{ine1}, we get $\mathcal{E}_{p,\Omega}(u)>0$. Then, from inequality (\ref{ine2}), we have $\Vert \nabla_\xi u_k\Vert_{L^p(\Omega)}\rightarrow\Vert \nabla_\xi u\Vert_{L^p(\Omega)}$  uniform on $\xi\in \mathbb{S}^{n-1}$, $\Vert \nabla_\xi u\Vert_{L^p(\Omega)}>c_1$ for every $\xi\in \mathbb{S}^{n-1}$ and $\Vert \nabla_\xi u_k\Vert_{L^p(\Omega)}^{-n}\leq c_2$, where $c_1,c_2>0$ are independents of $\xi$. Thus, $\Vert \nabla_\xi u_k\Vert_{L^p(\Omega)}^{-n}\rightarrow\Vert \nabla_\xi u\Vert_{L^p(\Omega)}^{-n}$ uniform on $\xi\in \mathbb{S}^{n-1}$ and so, by the dominated convergence theorem, we conclude $\mathcal{E}_{p,\Omega}(u_k)\rightarrow \mathcal{E}_{p,\Omega}(u)$.	This finish the proof.
	\end{proof}

The following lemma will be important to apply the general Brouwer's Fixed Point Theorem. This theorem states that if $E$ is a bounded and closed subset in $\mathbb{R}^m$ and homeomorphic to the closed unit ball, then any continuous mapping $S:E\rightarrow E$ admits a fixed point (see \cite{YL}).

\begin{lemma}\label{lemhomeo}
The affine ball $\overline{B}^m_\varrho(\zeta_0)=\{z \in \mathbb{R}^m; \lfloor z-\zeta_0\rfloor_{m}\leq \varrho\}$ and $\overline{B}^{p,m}_1(0)=\{z \in \mathbb{R}^m; \|z\|_{1,p,m}\leq 1\}$ are homeomorphic in $(\mathbb{R}^m, \|\cdot\|_{1,p,m})$.
\end{lemma}
\begin{proof}
Let us consider the applications $T: \overline{B}^m_\varrho(\zeta_0) \to \overline{B}^{p,m}_1(0)$ defined by
\[
T(x):=\left\{
\begin{array}{lcl}
\frac{x-\zeta_0}{\varrho}\frac{\lfloor x-\zeta_0\rfloor_{m}}{\|x-\zeta_0\|_{1,p,m}}	&\mbox{ if }& x\neq \zeta_0,\\
0 &\mbox{ if }& x=\zeta_0
\end{array}	
\right.
\]	
and $G:\overline{B}^{p,m}_1(0)  \to \overline{B}^m_\varrho(\zeta_0)$ defined by
\[
G(x):=\left\{
\begin{array}{lcl}
	\varrho x\frac{\|x\|_{1,p,m}}{\lfloor x\rfloor_{m}} +\zeta_0	&\mbox{ if }& x\neq 0,\\
	\zeta_0 &\mbox{ if }& x=0.
\end{array}	
\right.
\]	
By standard computations we obtain that $G=T^{-1}$, that is, $T\circ G = G\circ T=Id$. By the continuity of the functional $u \in W_m \mapsto \mathcal{E}_{p, \Omega}(u)$ follows the continuity of the function $T$ in $x\neq \zeta_0$ and $G$ in $x\neq 0$. Now, $T$ is continuous in $x=\zeta_0$ since applying inequality (\ref{ine1}), we have
\[
\|T(x)-T(\zeta_0)\|_{1,p,m}=\left\|\frac{x-\zeta_0}{\varrho}\frac{\lfloor x-\zeta_0\rfloor_{m}}{\|x-\zeta_0\|_{1,p,m}}\right\|_{1,p,m}=\frac{1}{\varrho}\lfloor x-\zeta_0\rfloor_{m}\leq \frac{1}{\varrho}\|x-\zeta_0\|_{1,p,m}.
\] 

To proof the continuity of $G$ in $x=0$, let us consider a sequence $(x_k)$ in $\overline{B}^{p,m}_1(0)$, with $x_k \to 0$. Lets suppose that the exists a subsequence $(x_{k_j})$ of $(x_k)$ such that 
\[
\lim_{j \to \infty}G(x_{k_j}) = \theta \neq \zeta_0.
\]  
Hence, by continuity of $T$, we obtain 
\[
\lim_{j \to \infty}x_{k_j} = \lim_{j \to \infty}T\circ G(x_{k_j})= T(\theta) \neq 0
\]
which contradicts $x_k \to 0$. In conclusion
\[
\lim_{k \to \infty}G(x_{k}) = \zeta_0 = G(0).
\]
Therefore, $T$ and $G$ are continuous, hence $\overline{B}^m_\varrho(\zeta_0)$ and $\overline{B}^{p,m}_1(0)$ are homeomorphic in $(\mathbb{R}^m, \|\cdot\|_{1,p,m})$.
\end{proof}

{\bf Proof of Theorem \ref{prop1}:}	Firstly, there exists $c(m)>0$ such that
	\begin{equation}\label{equiv.}
	    \lfloor x\rfloor_m=\mathcal{E}_{p,\Omega}\left(\sum_{j=1}^{m}x_jw_j\right)\leq \left\|\sum_{j=1}^{m}x_jw_j\right\|_{1,p,m}\leq c(m)|x|_2, \,\,\ \forall x \in \mathbb{R}^m.
	\end{equation}
	We assume $F(x)\neq0$ for all $x \in \overline{B}^{m}_\varrho(0)$. From \eqref{equiv.}, we have
	\[
	\overline{B}^{p,m}_\varrho(0):=\{z \in \mathbb{R}^m; \|z\|_{1,p,m}\leq \varrho\} \subset \overline{B}^m_\varrho(0). 
	\]
	We define
	\[
	g: (\overline{B}^{m}_\varrho(0), \|\cdot\|_{1,p,m}) \rightarrow (\mathbb{R}^m, \|\cdot\|_{1,p,m})
	\]
	by
	\[
	g(x)=-\frac{\varrho}{\lfloor F(x)\rfloor_{m}}F(x)
	\]
	which maps $\overline{B}^{m}_\varrho(0)$ into itself and is continuous. Therefore, by general Brouwer's Fixed Point Theorem, we have $g$ admits a fixed point $x_0$. Since $x_0=g(x_0)$, we obtain $\lfloor x_0\rfloor_{m}=\lfloor g(x_0)\rfloor_{m}=\varrho>0$. But then by \eqref{equiv.} and assumptions, we get
	\[
	0<\varrho^2=\lfloor x_0\rfloor^2_{m}\leq c(m)^2|x_0|^2_2=c(m)^2\left\langle x_0,x_0\right\rangle=c(m)^2\left\langle x_0,g(x_0)\right\rangle
	\]
	\[
	=-c(m)^2\frac{\varrho}{\lfloor F(x_0)\rfloor_{m}}\left\langle x_0,F(x_0)\right\rangle\leq 0,
	\]
	which is a contradiction.\\
	
{\bf Proof of Theorem \ref{prop2}:}	Suppose, $F(x)\neq0$ for all $x \in \overline{B}^{m}_\varrho(\zeta_0)$. Notice that by \eqref{equiv.} we have
	\[
	\overline{B}^{p,m}_\varrho(\zeta_0):=\{z \in \mathbb{R}^m; \|z-\zeta_0\|_{1,p,m}\leq \varrho\} \subset \overline{B}^m_\varrho(\zeta_0). 
	\]
	Define
	\[
	g_0: (\overline{B}^{m}_\varrho(\zeta_0), \|\cdot\|_{1,p,m}) \rightarrow (\mathbb{R}^m, \|\cdot\|_{1,p,m})
	\]
	by
	\[
	g_0(x)=-\frac{\varrho}{\lfloor F(x)\rfloor_{m}}F(x) +\zeta_0
	\]
	which maps $\overline{B}^{m}_\varrho(\zeta_0)$ into itself and is continuous. Hence it has a fixed point $x_0$, by general Brouwer's Fixed Point Theorem. Since $x_0=g_0(x_0)$, we have $\lfloor x_0-\zeta_0\rfloor_{m}=\lfloor g(x_0)-\zeta_0\rfloor_{m}=\varrho>0$. But then by \eqref{equiv.}
	\begin{eqnarray*}
	0<\varrho^2&=&\lfloor x_0-\zeta_0\rfloor^2_{m}\leq c(m)^2|x_0-\zeta_0|^2_2=c(m)^2\left\langle x_0-\zeta_0,x_0-\zeta_0\right\rangle\\
	&=& c(m)^2\left\langle x_0-\zeta_0,g_0(x_0)-\zeta_0\right\rangle=-c(m)^2\frac{\varrho}{\lfloor F(x_0)\rfloor_{m}}\left\langle x_0-\zeta_0,F(x_0)\right\rangle\leq 0,
	\end{eqnarray*}
	by assumptions, which is a contradiction.\\

{\bf Proof of Theorem \ref{prop3}:} The proof of this result follows analogous to that of Theorem \ref{prop2}, by using the equivalence between two norms in $\mathbb{R}^m$ and replacing $\lfloor \cdot\rfloor_{m}$ by $\Vert \cdot \Vert_m$.

\section{Application: existence of nontrivial critical point of the functional $\Phi_m$}\label{S3}

Before proving our application, we will show the following continuity result. Namely:

\begin{proposition} \label{T3} Let $A:\mathbb{R}^m \to \mathbb{R}^m$ be a function such that $$A(\zeta)=(A_1(\zeta),A_2(\zeta),\dots, A_m(\zeta)),$$ where $\zeta=(\zeta_1, \zeta_2, ..., \zeta_m) \in \mathbb{R}^m$,
	\[\begin{array}{lll}
	A_j(\zeta)=\int_\Omega H_{u}^{p-1}(\nabla u) \nabla H_{u}(\nabla u) \cdot \nabla w_j\; dx 
\end{array}	\]
	$j=1,2,\dots,m$, and $u=\sum_{i=1}^m\zeta_i w_i\in W_m$. Then, the function $A$ is continuous. 
	\end{proposition}
	
\begin{proof}
For this, let $\zeta^k\rightarrow\zeta^0$ in $\mathbb{R}^m$ and $u_k=\sum_{i=1}^m\zeta^k_i w_i\in W_m$, $k=0,1,2,\dots$. We fixed $j\in\{1,2,\dots,m\}$. Note that
\[
A_j(\zeta^k) = \gamma_{n,p}^{-n}\; \mathcal{E}_{p, \Omega}^{n + p}(u_k) \int_{\mathbb{S}^{n-1}} \Vert\nabla_\xi u_k\Vert_{L^p(\Omega)}^{-n-p}\int_{\Omega} \{\nabla_\xi u_k\}^{p-1}\langle \nabla w_j,\xi\rangle dxd\xi.
\]

If $\zeta^0= 0$, then by inequality (\ref{ine2}), we have $\Vert \nabla_\xi u_k\Vert_{L^p(\Omega)}\rightarrow 0$  uniform on $\xi\in \mathbb{S}^{n-1}$. Note that, using inequality (\ref{ine3}), there exists $D_0>0$ independent of $\xi$ such that
\[
\vert A_j(\zeta^k)\vert \leq \gamma_{n,p}^{-n}\; D_0 \int_{\mathbb{S}^{n-1}} \int_{\Omega} \vert\langle\nabla u_k,\xi\rangle\vert^{p-1}\vert \langle \nabla w_j,\xi\rangle\vert dxd\xi.
\]

Now, by H\"{o}lder and Cauchy-Schwarz inequalities and dominated convergence theorem, we have

\[
\vert A_j(\zeta^k)\vert \leq \gamma_{n,p}^{-n}\; D_0\vert\Omega\vert^{\frac{1}{p}} \Vert \nabla w_j\Vert_{L^p(\Omega)} \int_{\mathbb{S}^{n-1}} \Vert \nabla_\xi u_k\Vert^{p-1}_{L^p(\Omega)}d\xi\rightarrow 0=A_j(\zeta^0).
\]

If $\zeta^0\neq 0$, then by Lemma \ref{lema1}, we have $\mathcal{E}_{p, \Omega}^{n + p}(u_k)\rightarrow \mathcal{E}_{p, \Omega}^{n + p}(u)$. Finally, using the same arguments used in the proof of Lemma \ref{lema1} and in the proof of the previous case, we have $A_j(\zeta^k)\rightarrow A_j(\zeta^0)$. This concludes the proof.
\end{proof}

\begin{theorem}\label{TP2}
	Suppose that $f\in L^{p'}(\Omega)$ with $f\neq 0$ in $\Omega$, $1<\alpha<p$. Let $\Phi_m$ be the functional defined in (\ref{funct.}). Then there exists a nontrivial solution $u\in W_m$ satisfying $\Phi'_m(u)=0$.
\end{theorem}


\begin{proof}
Define the function $F:\mathbb{R}^m \to \mathbb{R}^m$ such that $$F(\zeta)=(F_1(\zeta),F_2(\zeta),\dots, F_m(\zeta)),$$ where $\zeta=(\zeta_1, \zeta_2, ..., \zeta_m) \in \mathbb{R}^m$,
	\[\begin{array}{lll}
	F_j(\zeta)&=&\displaystyle \int_{\Omega} H_{u}^{p-1}\left(\nabla u\right) \nabla H_{u}\left(\nabla u\right) \cdot \nabla w_j d x 
	 \\ && - \int_{\Omega}|u|^{\alpha - 2}uw_j dx  - \int_{\Omega}f(x)w_j dx, 
\end{array}	\]
	$j=1,2,\dots,m$, and $u=\sum_{i=1}^m\zeta_i w_i\in W_m$. 
	
By Proposition \ref{T3}, we have $F$ is a continuous function. Therefore, by H\"{o}lder's inequality, we get
	\[
	\begin{array}{lll}
	\left\langle F(\zeta),\zeta \right\rangle&=&\displaystyle  \mathcal{E}_{p, \Omega}^{p}(u)  - \|u\|^{\alpha}_{L^{\alpha}(\Omega)} - \int_{\Omega}f(x)u dx\\
	&\geq &	\mathcal{E}_{p, \Omega}^{p}(u)  - \|u\|^{\alpha}_{L^{\alpha}(\Omega)} - \|f\|_{L^{p'}(\Omega)}\|u\|_{L^{p}(\Omega)},
	\end{array}
	\]
	where $\frac{1}{p'}+\frac{1}{p}=1$. By affine Poincaré-Sobolev inequality (see inequality (\ref{APSob})), we obtain
	\[
	\begin{array}{lll}
	\left\langle F(\zeta),\zeta \right\rangle
	&\geq &	\mathcal{E}_{p, \Omega}^{p}(u)  - \left(\mu_{p,\alpha}^{\cal A}\right)^{-\alpha}\mathcal{E}_{p, \Omega}^{\alpha}(u) - \left(\mu_{p,p}^{\cal A}\right)^{-1}\|f\|_{L^{p'}(\Omega)}\mathcal{E}_{p, \Omega}(u).
	\end{array}
	\]

	Now, with the notations of Theorem \ref{prop1}, let  $\lfloor\zeta\rfloor_m=\varrho$ for some $\varrho>0$ to be taken below. Thus, we have
\[	\begin{array}{rcl}
\displaystyle\left\langle F(\zeta),\zeta\right\rangle  &\geq& \displaystyle \varrho^p - \left(\mu_{p,\alpha}^{\cal A}\right)^{-\alpha}\varrho^\alpha - \left(\mu_{p,p}^{\cal A}\right)^{-1}\|f\|_{L^{p'}(\Omega)}\varrho.
\end{array}\]

If $\varrho$ is such that
	\[
	\varrho>\max\left\{ \left[2\left(\mu_{p,p}^{\cal A}\right)^{-1}\|f\|_{L^{p'}(\Omega)}\right]^{\frac{1}{p-1}}, \left[2\left(\mu_{p,\alpha}^{\cal A}\right)^{-\alpha}\right]^{\frac{1}{p-\alpha}}\right\}
	\]
	then 
	\[
	\left\langle F(\zeta),\zeta\right\rangle >0.
	\]

We take 
\[
	\varrho:=\max\left\{ \left[2\left(\mu_{p,p}^{\cal A}\right)^{-1}\|f\|_{L^{p'}(\Omega)}\right]^{\frac{1}{p-1}}, \left[2\left(\mu_{p,\alpha}^{\cal A}\right)^{-\alpha}\right]^{\frac{1}{p-\alpha}}\right\}+1.
	\]	
	Thus,  Theorem \ref{prop1} ensure the existence of   $z_0 \in \mathbb{R}^m$ with $\lfloor y\rfloor_m\leq \varrho$ and such that $F(z_0)=0$. In other words, there exists $u_m \in W_m$ verifying
	\[
		\mathcal{E}_{p, \Omega}(u_m)\leq \varrho,
	\]
	and such that
	\[
	\displaystyle \int_{\Omega} H_{u_m}^{p-1}\left(\nabla u_m\right) \nabla H_{u_m}\left(\nabla u_m\right) \cdot \nabla w d x = \int_{\Omega}|u_m|^{\alpha - 2}u_mw dx +\int_{\Omega}f(x)w dx 
	\]
	for all  $w \in W_m$. Then $u_m$ is nontrivial. In particular, considering $w=u_m$ we obtain
	\begin{eqnarray*} 
	\mathcal{E}_{p, \Omega}^{p}(u_m)   &=& \|u_m\|^{\alpha}_{L^\alpha(\Omega)}+ \int_{\Omega}f(x)u_m dx.
	\end{eqnarray*}
	
	\begin{remark}\label{rem1}
    It is important to mention that $\varrho$, does not depend on $m$.
\end{remark}
	
Thus, $\mathcal{E}_{p, \Omega}^{p}(u_m)$ is bounded and so, by Theorem 6.5.3 of
\cite{T1}, there exists $u \in W^{1,p}_0(\Omega)$ such that $u_m \rightarrow u$ strongly in $L^s(\Omega)$ once $s < p^*$, up to a subsequence.

\end{proof} 

\subsection*{Acknowledgements}

The first author was partially supported by FAPEMIG/ APQ-02375-21, APQ-04528-22, FAPEMIG/RED-00133-21 and CNPq Process 307575/2019-5 and 101896/2022-0.\\
	The second author was partially supported by CNPq/Brazil (PQ 316526/2021-5) and Fapemig/Brazil (Universal-APQ-00709-18).


\begin{thebibliography}{}
	%
	%
	%
	


\bibitem{af1}
{C. O. Alves, D. G. de Figueiredo} - \textit{Nonvariational Elliptic Systems via Galerkin Methods}, D. Haroske, T. Runst and H. J. Schmeisser (eds.) Function Spaces, Differential Operators and Nonlinear Analysis. The Hans Triebel Anniversary Volume, 2003.

%
%






\bibitem{Brezis}
{H.  Br\'{e}zis} - \textit{Functional  Analysis},   Sobolev  Spaces  and  Partial  Differential   Equations,   585   DOI   10.1007/978-0-387-70914-7, Springer   Science Business Media, LLC 2011.


%
%
%
%
%
%

	

\bibitem{AL} 
{A. L. A. de Araujo, L. F. O. Faria} - \textit{Positive solutions of quasilinear elliptic equations with exponential nonlinearity combined with convection term}, J. Differential Equations 267 (2019), 4589-4608.

\bibitem{Araujo-F2}
A. L. A. de Araujo, L. F. O. Faria - \emph{Existence, nonexistence, and asymptotic behavior of solutions for $N$-Laplacian equations involving critical exponential growth in the whole $\mathbb{R}^N$}. Math. Ann. 384 (2022), 1469-1507.

\bibitem{NHJM}
{P. L. De Nápoli, J. Haddad, C. H. Jiménez, M. Montenegro} - \textit{The sharp affine $L^2$ Sobolev trace inequality and variants}, Math. Ann. 370 (2018), 287-308.

	
	
%
%
%



%
%
	
	
	
	
	

	\bibitem{FJN}
	{S. Fu\v{c}\'{\i}k, O. John, J. Ne\v{c}as} - \textit{On the existence of Schauder
	bases in Sobolev spaces}, Comment. Math. Univ. Carolinae 13 (1972), 163-175.

%


	

%


\bibitem{HS}
{C. Haberl, F. E. Schuster} - \textit{Asymmetric affine $L_p$ Sobolev inequalities}, J. Funct. Anal. 257 (2009), 641-658.

\bibitem{HJM1}
{J. Haddad, C. H. Jiménez, M. Montenegro} - \textit{Sharp affine Sobolev type inequalities via the $L_p$ Busemann-Petty centroid inequality}, J. Funct. Anal. 271 (2016), 454-473.

\bibitem{HJM3}
{J. Haddad, C. H. Jiménez, M. Montenegro} - \textit{Sharp affine weighted $L^p$ Sobolev type inequalities}, Trans. Amer. Math. Soc. 372 (2019), 2753-2776.

\bibitem{HJM4}
{J. Haddad, C. H. Jiménez, M. Montenegro} - \textit{Asymmetric Blaschke-Santaló functional inequalities}, J. Funct. Anal. 278 (2020), 108319, 18 pp.

\bibitem{HJM5}
{J. Haddad, C. H. Jiménez, M. Montenegro} - \textit{From affine Poincaré inequalities to affine spectral inequalities}, Adv. Math. 386 (2021), 107808, 35 pp.

\bibitem{HJS}
{J. Haddad, C. H. Jiménez, L. A. Silva} - \textit{An $L_p$-functional Busemann-Petty centroid inequality}, Int. Math. Res. Not. 2021 (2021), 7947-7965.

	

	

	\bibitem{k}
	{S. Kesavan} - \textit{Topics in functional analysis and applications}, John Wiley $\&$ Sons (1989).
	


%
%
%
%
	\bibitem{LM1}
{E. J. F. Leite, M. Montenegro} - \textit{Minimization to the Zhang's energy on $BV(\Omega)$ and sharp affine Poincaré-Sobolev inequalities}, J. Funct. Anal. 283 (2022), no. 10, Paper No. 109646.
	
	\bibitem{LM}
{E. J. F. Leite, M. Montenegro} - \textit{Least energy solutions for affine $p$-Laplace equations involving subcritical and critical nonlinearities}, arXiv:2202.07030v2.

\bibitem{LM2}
{E. J. F. Leite, M. Montenegro} - \textit{Towards existence theorems to affine $p$-Laplace equations via a new variational framework}, preprint.


%
%
%
%
%
\bibitem{LXZ}
{M. Ludwig, J. Xiao, G. Zhang} - \textit{Sharp convex Lorentz-Sobolev inequalities}, Math. Ann. 350 (2011), 169-197.


\bibitem{LYZ2}
{E. Lutwak, D. Yang, G. Zhang} - \textit{Sharp affine $L_p$ Sobolev inequalities}, J. Differential Geom. 62 (2002), 17-38.

\bibitem{LYZ3}
{E. Lutwak, D. Yang, G. Zhang} - \textit{Optimal Sobolev norms and the $L^p$ Minkowski problem}, Int. Math. Res. Not. 2006 (2006), 62987.
%
%
%
%
%
%
%
	
	
	
	

	
	 

	
	
	










\bibitem{s}
{W. A. Strauss} - \textit{On weak solutions of semilinear hyperbolic equations}, An. Acad. Brasil. Ci\^enc. 42 (1970), 645-651.

\bibitem{T1}
{C. Tintarev} - \textit{Concentration compactness: Functional-analytic theory of concentration phenomena}, Berlin, Boston: De Gruyter, 2020. https://doi.org/10.1515/9783110532432.



%
%
\bibitem{YL}
{B. Yu, Z. Lin} - \textit{Homotopy method for a class of nonconvex Brouwer fixed-point problems}. Applied Mathematics and Computation 74.1 (1996), 65-77.
%





	
\end{thebibliography}


\end{document}